\documentclass[12pt]{article}
\usepackage[a4paper, tmargin=1.0in, bmargin=1.0in,lmargin=1in,rmargin=1in]{geometry}
\usepackage{amssymb}
\usepackage{amsthm}
\usepackage{mathptmx}
\usepackage{times}
\usepackage{setspace}
\usepackage{url}
\usepackage[fleqn]{amsmath}
\usepackage{graphicx,epsfig}
\usepackage[pdfpagelabels=true,plainpages=false,colorlinks=true,linkcolor=blue,citecolor=blue,urlcolor=blue]{hyperref}
\theoremstyle{definition}

\newtheorem{theorem}{Theorem}

\newtheorem{corollary}{Corollary}
\newtheorem{proposition}{Proposition}
\theoremstyle{lemma}
\newtheorem{remark}{Remark}
\theoremstyle{remark}
\newtheorem{lem}{Lemma}[section]
\date{}

\begin{document}
\title{Certain geometric structure of $\Lambda$-sequence spaces}
\author{ Atanu Manna \footnote{Author's e-mail: atanu.manna@iict.ac.in/atanumanna@maths.iitkgp.ernet.in}\\
\textit{\small{Indian Institute of Carpet Technology}} \\
\textit{\small{Chauri Road, Bhadohi- 221401, Uttar Pradesh, India}}}
\maketitle
\begin{abstract}
The $\Lambda$-sequence spaces $\Lambda_p$ for $1< p\leq\infty$ and its generalization $\Lambda_{\hat{p}}$ for $1<\hat{p}<\infty$, $\hat{p}=(p_n)$ is introduced. The James constants and strong $n$-th James constants of $\Lambda_p$ for $1<p\leq\infty$ is determined. It is proved that generalized $\Lambda$-sequence space $\Lambda_{\hat{p}}$ is embedded isometrically in the Nakano sequence space $l_{\hat{p}}(\mathbb{R}^{n+1})$ of finite dimensional Euclidean space $\mathbb{R}^{n+1}$. Hence it follows that sequence spaces $\Lambda_p$ and $\Lambda_{\hat{p}}$ possesses the uniform Opial property, property $(\beta)$ of Rolewicz and weak uniform normal structure. Moreover, it is established that $\Lambda_{\hat{p}}$ possesses the coordinate wise uniform Kadec-Klee property. Further necessary and sufficient conditions for element $x\in S(\Lambda_{\hat{p}})$ to be an extreme point of $B(\Lambda_{\hat{p}})$ are derived. Finally, estimation of von Neumann-Jordan and James constants of two dimensional $\Lambda$-sequence space $\Lambda_2^{(2)}$ is being carried out.
\end{abstract}
\section{Introduction}
There are several important geometric constants of Banach spaces such as von Neumann-Jordan constant, James constant, Dunkl-Williams constant, Khintchine constant, Zb\u{a}ganu, Ptolemy constants and so on. The space is how much close (or far) to a Hilbert space measured by von Neumann-Jordan constant and Dunkl-Williams constant. The uniform non-squareness of a unit ball in a real Banach space measured by James constant (or sometimes called \emph{James non-square constant}).
These constants occupied a prominent place in the study of geometrical properties of Banach spaces, and have investigated recently by many researchers (see e.g. \cite{FUST}, \cite{MAL1}, \cite{MAL2}).\\
Maligranda et al. \cite{MAL1} obtained the classical James constant and $n$-th James constant for the Ces\`{a}ro sequence spaces $ces_p$, $1<p\leq\infty$, which was first introduced and studied by Shiue \cite{SHUE} and Leibowitz \cite{LEIB}. Let $l^0$ be the space of all real sequences and $\mathbb{N}_0$ be the set of all natural numbers $\mathbb{N}$ including $0$, i.e., $\mathbb{N}_0=\mathbb{N}\cup \{0\}$ and $x=(x_n)_{n\in \mathbb{N}_0}\in l^0$. Throughout the text we shall write $x=(x_n)$ instead of $x=(x_n)_{n\in \mathbb{N}_0}$. Then the Ces\`{a}ro sequence spaces $ces_p$, $1<p\leq\infty$ are defined as follows:
\begin{center}
$ces_p=\Big\{x\in l^0:  \Big(\displaystyle\sum_{n=0}^{\infty}\Big(\frac{1}{n+1}\sum_{k=0}^{n}|x_k|\Big)^p\Big)^{\frac{1}{p}}<\infty\Big\}$ for $1< p<\infty$\\
$ces_\infty=\Big\{x\in l^0:  \displaystyle\sup_{n\in \mathbb{N}_0}\frac{1}{n+1}\sum_{k=0}^{n}|x_k|<\infty\Big\}$.
\end{center}
The Ces\`{a}ro sequence space $ces_p$ is generalized to $ces_{\hat{p}}$ for $\hat{p}=(p_n)$, $p_n > 1$, $n\in \mathbb{N}_0$ (\cite{JOHN}) and defined as
\begin{center}
$ces_{\hat{p}}=\Big\{x\in l^0:  \Big(\displaystyle\sum_{n=0}^{\infty}\Big(\frac{1}{n+1}\sum_{k=0}^{n}|x_k|\Big)^{p_n}\Big)^{\frac{1}{p_n}}<\infty\Big\}$.
\end{center}
Several authors (\cite{cui2000banachsaks}, \cite{CUIHUD}, \cite{CUI}, \cite{FORA}, \cite{MAN}, \cite{PET1}, \cite{PET2}, \cite{SAEJ}, \cite{SANH}, \cite{SANH1}, \cite{SUAN}, \cite{SUAN1}) studied geometric properties such as Opial property, Kadec-Klee property, property $(\beta)$ of Rolewicz, weak uniform normal structure etc. for the spaces $ces_p$ and $ces_{\hat{p}}$. These constants play very important role in the study of fixed point theory. For example, Opial property has several applications in the Banach fixed point theory, differential equations, integral equations etc. On the other hand, Kadec-Klee property applied to established certain results in the ergodic theory (see \cite{PRUS}).\\
In this present paper, we are willing to study certain geometric structure, as the title said, of $\Lambda$-sequence spaces $\Lambda_p$, $1<p\leq\infty$ and its generalization $\Lambda_{\hat{p}}$, for $1< \hat{p}<\infty$, $\hat{p}=(p_n)$. These spaces $\Lambda_p$ and $\Lambda_{\hat{p}}$ are not only gives the sequence spaces $ces_p$ and $ces_{\hat{p}}$ but also include sequence spaces generated by Riesz weighted means, N\"{o}rlund means etc. \cite{BOOS} in special cases. For example, our generalization include the following important results:\\
$(i)$ $ces_p$ and $ces_{\hat{p}}$ has the uniform Opial property (see \cite{CUIHUD} and \cite{PET1} respectively),\\
$(ii)$ $ces_p$ and $ces_{\hat{p}}$ has the property $(\beta)$ (see \cite{cui2000banachsaks} and \cite{SAEJ} respectively),\\
$(iii)$ $ces_p$ has weak uniform normal structure \cite{CUIHUD} and $ces_{\hat{p}}$ has the same \cite{SAEJ},\\
$(iv)$ $ces_{\hat{p}}$ possesses the uniform Kadec Klee property \cite{PET2} and Kadec-Klee property (\cite{SUAN}, \cite{SUAN1}),\\
$(v)$ The James constants $J(ces_p)=2$, $J_n^s(ces_p)=n$ for $1<p\leq \infty$ and $J(ces_2^{(2)})=\sqrt{2+\frac{2}{\sqrt{5}}}$ \cite{MAL1},\\
$(vi)$ $ces_p$ and $ces_{\hat{p}}$ has extreme point.\\
Therefore studying geometric structure of the spaces $\Lambda_p$ and $\Lambda_{\hat{p}}$ is a unified study of geometric structure for the known sequence spaces.
\subsection{Sequence spaces $\Lambda_p$, $1< p\leq\infty$}
The notions of $\Lambda$-strong convergence was first introduced by M\'{o}ricz \cite{MORI} and is behind the genesis of $\Lambda$-sequence spaces. Let $\Lambda=\{\lambda_k: k=0, 1, 2, \ldots\}$ be a non-decreasing sequence of positive numbers tending to $\infty$, i.e., $0<\lambda_0<\lambda_1<\lambda_2<\ldots$, $\lambda_k\rightarrow\infty$ with $\frac{\lambda_{k+1}}{\lambda_k}\rightarrow 1$ as $k\rightarrow\infty$ and $x=(x_k)\in l^0$.\\
Define sequence spaces $\Lambda_p$, for $1< p \leq\infty$ as
\begin{center}
$\Lambda_p=\big\{x=(x_k)\in l^0:  \|x\|_p<\infty\big\}$ for $1< p<\infty$\\
$\Lambda_\infty=\big\{x=(x_k)\in l^0:  \|x\|_\infty<\infty\big\}$ for $p=\infty$,
\end{center}
where $\|.\|_p$, $1< p<\infty$ and $\|x\|_\infty$ are defined by
\begin{align*}
\|x\|_p &= \Big(\displaystyle\sum_{n=0}^{\infty}\big(\frac{1}{\lambda_n}\sum_{k=0}^{n}(\lambda_k-\lambda_{k-1})|x_k|\big)^p\Big)^{\frac{1}{p}} & \mbox{~and~}
\|x\|_\infty &= \displaystyle\sup_{n\in \mathbb{N}_0}\frac{1}{\lambda_n}\sum_{k=0}^{n}(\lambda_k-\lambda_{k-1})|x_k|.
\end{align*}
It is a routine work to established that these spaces $(\Lambda_p, \|.\|_p)$ for $1<p<\infty$ and $(\Lambda_\infty, \|.\|_\infty)$ are Banach spaces.
\subsection{Sequence spaces $\Lambda_{\hat{p}}$, $\hat{p}=(p_n)$, $p_n>1$}
Let $\hat{p}=(p_n)$ be a bounded sequence of positive real numbers such that $p_n>1$ for each $n\in \mathbb{N}_0$. Define a convex modular $\sigma(x)$ on $l^0$ as $\sigma(x)=\displaystyle\sum_{n=0}^{\infty}\Big(\frac{1}{\lambda_n}\sum_{k=0}^{n}(\lambda_k-\lambda_{k-1})|x_k|\Big)^{p_n}$ and denote $\Lambda x(n)=\frac{1}{\lambda_n}\displaystyle\sum_{k=0}^{n}(\lambda_k-\lambda_{k-1})|x_k|$.\\
Then we define the following set
\begin{center}
$\Lambda_{\hat{p}}=\big\{x=(x_k)\in l^0:  \sigma(r x)<\infty \mbox{~for some~} r>0\big\}$,
\end{center} which is a normed linear space equipped with the Luxemberg norm
$$\|x\|_{\hat{p}}= \inf \Big\{ r>
0 : \sigma\Big(\frac {x} {r}\Big) \leq 1 \Big\}.$$
Indeed $(\Lambda_{\hat{p}}, \|x\|_{\hat{p}})$ becomes a Banach space.
\begin{remark}
In particular, \\
$(i)$ if we put $\lambda_n=n+1$, then sequence spaces $\Lambda_p$ and $\Lambda_{\hat{p}}$ reduces to $ces_p$ and $ces_{\hat{p}}$ respectively (\cite{JOHN}, \cite{LEIB} and \cite{SHUE}). \\
$(ii)$ choose $q=(q_k)=(\lambda_k-\lambda_{k-1})$ and $Q_n=\displaystyle\sum_{k=0}^{n}q_k=\lambda_n$, then sequence spaces $\Lambda_p$ and $\Lambda_{\hat{p}}$ reduces to $ces[p, q]$ and $ces[{\hat{p}}, {\hat{q}}]$ respectively \cite{JOHN}.\\
$(iii)$ when $p_{n-k}=(\lambda_k-\lambda_{k-1})$ for $k=0, 1, \ldots, n$ and $\overline{N}=\displaystyle\sum_{k=0}^{n}q_k=\lambda_n$, then we get absolute type N\"{o}rlund sequence spaces. The N\"{o}rlund sequence spaces of non-absolute type is studied by Wang \cite{WANG}.
\end{remark}
\section{James constants of $\Lambda_p$ for $1< p \leq\infty$}
The unit ball of a normed linear space is \emph{uniformly non-square} if and only if there is a positive number $\delta$ such that there do not exist members $x$ and $y$ of the unit ball for which $\|\frac{1}{2}(x+y)\|>1-\delta$ and $\|\frac{1}{2}(x-y)\|>1-\delta$ (see \cite{JAMES}).\\
The \emph{James constant} (or \emph{measure of uniform non-squareness}) of a real Banach space $(X, \|\cdot\|)$ with $dim(X)\geq 2$ is denoted by $J(X)$ and is defined as (\cite{DIES}, \cite{GAO})
\begin{center}
$J(X)=\sup\{\min(\|x+y\|, \|x-y\|):~ x, y\in X,~ \|x\|=1,~ \|y\|=1\}.$
\end{center}
We begin with finding the James constants of the sequence spaces $\Lambda_p$ for $1<p\leq\infty$.
\begin{theorem}{\label{thm1}}
The James constants of $\Lambda$-sequence spaces $\Lambda_p$ for $1<p\leq\infty$ is $2$. In notation $J(\Lambda_p)=2$, $1< p\leq \infty$.
\end{theorem}
\begin{proof}
First we discuss the case $1< p<\infty$. For each $m=0, 1, 2, \ldots$, we denote $e_m=(e_{mk})=(0, 0, \ldots, 0, 1, 0, \ldots)$, where $1$ is at the $m$-th position. Note that
\begin{align*}
\frac{1}{\lambda_n}\displaystyle\sum_{k=0}^{n}(\lambda_k-\lambda_{k-1})e_{mk} &= \left\{
\begin{array}{ll}
0 & \quad \mbox{~if~} n<m,\\
\frac{\lambda_m-\lambda_{m-1}}{\lambda_n} & \quad \mbox{~if~} n\geq m.
\end{array}\right.
\end{align*}
Let $x_m=(x_{mk})$ and $y_m=(y_{mk})$, where $x_m=\frac{e_m}{\|e_m\|_p}$ and $y_m=\frac{e_{m+1}}{\|e_{m+1}\|_p}$ for each $m=0, 1, 2, \ldots$.
Then $\|x_m\|_p=1$ and $\|y_m\|_p=1$.\\
Now we find the value of $\left\| x_m\pm y_m\right\|_p$ for $m=0, 1, 2, \ldots$. We have
\begin{align} {\label{eqn1}}
&\|x_m\pm y_m\|_p^p \nonumber\\
&=\displaystyle\sum_{n=0}^{\infty}\Big(\frac{1}{\lambda_n}\displaystyle\sum_{k=0}^{n}(\lambda_k-\lambda_{k-1})|x_{mk}\pm y_{mk}|\Big)^p \nonumber\\
&=\displaystyle\sum_{n=0}^{\infty}\Big(\frac{1}{\lambda_n}\displaystyle\sum_{k=0}^{n}(\lambda_k-\lambda_{k-1})\Big|\frac{e_{mk}}{\|e_m\|_p}\pm \frac{e_{(m+1)k}}{\|e_{m+1}\|_p}\Big|~\Big)^p \nonumber \\
&=\Big(\frac{\lambda_m-\lambda_{m-1}}{\lambda_m\|e_m\|_p}\Big)^p+ \displaystyle\sum_{n=m+1}^{\infty}\Big(\frac{1}{\lambda_n}
\Big(\frac{\lambda_m-\lambda_{m-1}}{\|e_m\|_p}+\frac{\lambda_{m+1}-\lambda_{m}}{\|e_{m+1}\|_p}\Big)\Big)^p \nonumber \\
&\geq \displaystyle\sum_{n=m+1}^{\infty}\Big(\frac{1}{\lambda_n}\Big(\frac{\lambda_m-\lambda_{m-1}}{\|e_m\|_p}+\frac{\lambda_{m+1}-\lambda_{m}}
{\|e_{m+1}\|_p}\Big)\Big)^p \nonumber \\
&=\Big(1+\frac{\lambda_m-\lambda_{m-1}}{\lambda_{m+1}-\lambda_{m}}.\frac{\|e_{m+1}\|_p}{\|e_{m}\|_p}\Big)^p.
\end{align}
Therefore by removing the power $p$, we get $\|x_m\pm y_m\|_p\geq 1+\frac{\lambda_m-\lambda_{m-1}}{\lambda_{m+1}-\lambda_{m}}.\frac{\|e_{m+1}\|_p}{\|e_{m}\|_p}$. We claim that $$\displaystyle\lim_{m\rightarrow\infty}\frac{\lambda_m-\lambda_{m-1}}{\lambda_{m+1}-\lambda_{m}}.\frac{\|e_{m+1}\|_p}{\|e_m\|_p}=1.$$
Since
$$\Big(\frac{\lambda_m-\lambda_{m-1}}{\lambda_{m+1}-\lambda_{m}}\Big)^p.\frac{\|e_{m+1}\|_p^p}{\|e_m\|_p^p}=
\frac{\displaystyle\sum_{n=m+1}^{\infty}\frac{1}{{\lambda_n^p}}}{\displaystyle\sum_{n=m}^{\infty}\frac{1}{{\lambda_n^p}}}
=\frac{\displaystyle\sum_{n=m+1}^{\infty}\frac{1}{{\lambda_n^p}}}{\frac{1}{{\lambda_m^p}}+\displaystyle\sum_{n=m+1}^{\infty}
\frac{1}{{\lambda_n^p}}}
=\Bigg(1+ \frac{\frac{1}{{\lambda_m^p}}}{\displaystyle\sum_{n=m+1}^{\infty}\frac{1}{{\lambda_n^p}}}\Bigg)^{-1},$$
by the discrete version of Bernoulli-de l' Hospital rule, we have
\begin{equation}{\label{eqn2}}
\displaystyle\lim_{m\rightarrow\infty}\frac{\frac{1}{{\lambda_m^p}}}{\displaystyle\sum_{n=m+1}^{\infty}
	\frac{1}{{\lambda_n^p}}}
=\lim_{m\rightarrow\infty}\frac{\frac{1}{{\lambda_m^p}}-\frac{1}{{\lambda_{m+1}^p}}}{\frac{1}{{\lambda_{m+1}^p}}}
=\lim_{m\rightarrow\infty}\frac{\frac{1}{{\lambda_m}^p}}{\frac{1}{{\lambda_{m+1}^p}}}-1=\lim_{m\rightarrow\infty}\Big(\frac{\lambda_{m+1}}{{\lambda_m}}\Big)^p-1=0
\end{equation}
Eqn (\ref{eqn2}) proves our claim. Therefore from Eqn (\ref{eqn1}), we have $\|x_m\pm y_m\|_p\geq 2$ as $m\rightarrow\infty$. Since we always know that   $\|x_m\pm y_m\|_p\leq 2$, so $J(\Lambda_p)=2$ for $1<p<\infty$.\\
For $p=\infty$, we choose $x_m=\frac{\lambda_m}{\lambda_m-\lambda_{m-1}}e_m$ and $y_m=\frac{\lambda_m}{\lambda_{m+1}-\lambda_m}e_{m+1}$. Then it is easy to verify that $\|x_m\|_{\infty}=1$, $\|y_m\|_{\infty}=1$ and $\|x_m\pm y_m\|_{\infty}=1+\frac{\lambda_m}{\lambda_{m+1}}\rightarrow2$ as $m\rightarrow\infty$. Hence $J(\Lambda_\infty)=2$.
\end{proof}
\begin{corollary}
$(i)$ Choose $\lambda_n=n+1$, then $J(ces_p)=2$ for $1<p\leq \infty$ (\cite{MAL1}).\\
$(ii)$ If $Q_n=\displaystyle\sum_{k=0}^{n}q_k=\lambda_n$, where $q=(q_k)=(\lambda_k-\lambda_{k-1})$ then $J(ces [p, q])=2$ for $1<p\leq \infty$.
\end{corollary}
A Banach space is said to be \emph{uniformly non-$l_n^{(1)}$} if there is $\delta\in (0,1)$ such that for any $x_0, x_1, x_2, \ldots, x_{n-1}$ from the unit ball of $X$, we have $\displaystyle\min_{\epsilon_k=\pm 1}\big\|\sum_{k=0}^{n-1}\epsilon_kx_k\big\|\leq n(1-\delta)$.\\
This definition leads to the notion of \emph{$n$-th James constant} (or the \emph{measure of uniformly non-$l_n^{(1)}$}) $J_n(X)$, $n\in \mathbb{N}$ of a Banach space $X$ is defined as
\begin{equation}
J_n(X)=\sup\Big\{\displaystyle\min_{\varepsilon_k=\pm 1}\big\|\sum_{k=0}^{n-1}\epsilon_kx_k\big\|:~ x_k\in X,~ \|x_k\|\leq1, k=0, 1, 2, \ldots, {n-1}\Big\} ~\cite{DIES}.\nonumber
\end{equation}
When we restrict to unit sphere of a real Banach space $X$ then the James constants (or \emph{$n$-th strong James constants}) are denoted by $J_n^s(X)$, $n\in \mathbb{N}$ and defined by
\begin{center}
$J_n^s(X)=\sup\Big\{\displaystyle\min_{\varepsilon_j=\pm1}\Big\|\displaystyle\sum_{j=0}^{n-1}\varepsilon_jx_j\Big\|: \|x_j\|=1, j=0, 1, 2, \ldots, n-1\Big\}.$
\end{center}
It is to be noted that $J_n^s(X)\leq J_n(X)\leq n$ and $J_2^s(X)=J_2(X)=J(X)$ \cite{MAL1}.
\begin{theorem}
The strong $n$-th James constant $J_n^s(\Lambda_p)=n$ for $1<p<\infty$, and $J_n^s(\Lambda_\infty)=n$.
\end{theorem}
\begin{proof}
In previous theorem, we have established the result for the case when $n=2$. Therefore we deduce this result for $n\geq 3$. Let $m \in \mathbb{N}_0$, $n\in \mathbb{N}$ such that $n\geq 3$ and put,
$$x_{j,m}=\frac{e_{m+j}}{\|e_{m+j}\|_p}, j=0, 1, 2, \ldots, n-1.$$
With this setting it is clear that $\|x_{j,m}\|_p=1$ for each $j=0, 1, 2, \ldots, n-1$. Since $x_{j_1,m}$ and $x_{j_2,m}$ have disjoint supports for $j_1\neq j_2$, so it implies that $$\displaystyle\min_{\varepsilon_j=\pm1}\Big\|\displaystyle\sum_{j=0}^{n-1}\varepsilon_jx_{j,m}\Big\|_p=
\Big\|\displaystyle\sum_{j=0}^{n-1}x_{j,m}\Big\|_p.$$
Choose $b_{m,l}=\displaystyle\sum_{i=m}^{m+l}\frac{\lambda_i-\lambda_{i-1}}{\|e_i\|_p}$, $l=0, 1, 2, \ldots, n-1$, then we have
\begin{align*}
\Big\|\displaystyle\sum_{j=0}^{n-1}x_{j,m}\Big\|_p^p
&=\Big\|\displaystyle\sum_{j=m}^{m+n-1}\frac{e_j}{\|e_j\|_p}\Big\|_p^p\\
&=\Big\|\Big(0, 0, \ldots, 0, \frac{1}{\|e_m\|_p}, \frac{1}{\|e_{m+1}\|_p}, \ldots, \frac{1}{\|e_{m+n-1}\|_p}, 0, 0, \ldots\Big)\Big\|_p^p\\
&=\Big(\frac{b_{m,0}}{\lambda_m}\Big)^p+ \Big(\frac{b_{m,1}}{\lambda_{m+1}}\Big)^p+\ldots+\Big(\frac{b_{m,n-1}}{\lambda_{m+n-1}}\Big)^p+\Big(\frac{b_{m,n-1}}{\lambda_{m+n}}\Big)^p+\ldots\\
&\geq \displaystyle\sum_{k=m+n-1}^{\infty}\Big(\frac{b_{m,n-1}}{\lambda_{k}}\Big)^p\\
&=(b_{m,n-1})^p\displaystyle\sum_{k=m+n-1}^{\infty}\Big(\frac{1}{\lambda_{k}}\Big)^p\\
&=\Big(\displaystyle\sum_{i=m}^{m+n-1}\frac{\lambda_i-\lambda_{i-1}}{\|e_i\|_p}\Big)^p.\Big(\frac{\|e_{m+n-1}\|_p}{\lambda_{m+n-1}-\lambda_{m+n-2}}\Big)^p\\
& \geq n^p\Big(\frac{\lambda_m-\lambda_{m-1}}{\|e_m\|_p}\Big)^p.\Big(\frac{\|e_{m+n-1}\|_p}{\lambda_{m+n-1}-\lambda_{m+n-2}}\Big)^p
\end{align*}
Consequently, we have
\begin{equation}{\label{eqn3}}
n\geq\Big\|\displaystyle\sum_{j=0}^{n-1}x_{j,m}\Big\|_p\geq n \frac{\lambda_m-\lambda_{m-1}}{\lambda_{m+n-1}-\lambda_{m+n-2}}\frac{\|e_{m+n-1}\|_p}{\|e_m\|_p}.
\end{equation}
We claim that $\displaystyle\lim_{m\rightarrow\infty}\frac{\lambda_m-\lambda_{m-1}}{\lambda_{m+n-1}-\lambda_{m+n-2}}.\frac{\|e_{m+n-1}\|_p}{\|e_m\|_p}=1$. Denote $a_{m+n-1}:=\displaystyle\sum_{j=m+n-1}^{\infty}\frac{1}{{\lambda_j^p}}$.\\
Then we have \begin{align*}
&\Big(\frac{\lambda_m-\lambda_{m-1}}{\lambda_{m+n-1}-\lambda_{m+n-2}}\Big)^p.\frac{\|e_{m+n-1}\|_p^p}{\|e_m\|_p^p}
=\frac{\displaystyle\sum_{j=m+n-1}^{\infty}\frac{1}{{\lambda_j^p}}}{\displaystyle\sum_{j=m}^{\infty}\frac{1}{{\lambda_j^p}}}
=\frac{a_{m+n-1}}{\frac{1}{{\lambda_m^p}}+\frac{1}{{\lambda_{m+1}^p}}+\ldots+ \frac{1}{{\lambda_{m+n-2}^p}}+ a_{m+n-1}}\\
&=\Big(1+\frac{1}{{\lambda_m^pa_{m+n-1}}}+\frac{1}{{\lambda_{m+1}^pa_{m+n-1}}}+\ldots+ \frac{1}{{\lambda_{m+n-2}^p}a_{m+n-1}}\Big)^{-1}.
\end{align*}
Now we find out the limits $\displaystyle\lim_{m\rightarrow\infty}\frac{1}{{\lambda_{m+i}^pa_{m+n-1}}}$ for $i=0, 1, \ldots, n-1$. When $i=0$, applying Bernoulli-de l' Hospital rule, we get
\begin{equation}{\label{eqn4}}
\displaystyle\lim_{m\rightarrow\infty}\frac{\frac{1}{{\lambda_m^p}}}{\displaystyle\sum_{j=m+n-1}^{\infty}
	\frac{1}{{\lambda_j^p}}}
=\lim_{m\rightarrow\infty}\frac{\frac{1}{{\lambda_m^p}}-\frac{1}{{\lambda_{m+1}^p}}}{\frac{1}{{\lambda_{m+n-1}^p}}}
=\lim_{m\rightarrow\infty}\Big(\frac{\lambda_{m+n-1}}{{\lambda_m}}\Big)^p-\lim_{m\rightarrow\infty}\Big(\frac{\lambda_{m+n-1}}{{\lambda_{m+1}}}\Big)^p=0.
\end{equation} Therefore Eqn (\ref{eqn4}) proved our claim. Similarly, we get $\displaystyle\lim_{m\rightarrow\infty}\frac{1}{{\lambda_{m+i}^pa_{m+n-1}}}=0$ for $i= 1, \ldots, n-1$. Hence by Eqn (\ref{eqn3}), we get $\Big\|\displaystyle\sum_{j=0}^{n-1}x_{j,m}\Big\|_p\geq n$ as $m\rightarrow\infty$.\\
Since $\Big\|\displaystyle\sum_{j=0}^{n-1}x_{j,m}\Big\|_p\leq n$, so by definition we obtain $J_n^s(\Lambda_p)=n$ for $1<p<\infty$.\\
For the case $p=\infty$, we choose $$x_{j,m}=\frac{\lambda_{m+j}}{\lambda_{m+j}-\lambda_{m+j-1}}e_{m+j}, \mbox{~for~} j=0, 1, \ldots,n-1.$$
It is easy to find that $\|x_{j,m}\|_{\infty}=1$ and for any fixed $n\in \mathbb{N}$
\begin{align*}
\Big\|\displaystyle\sum_{j=0}^{n-1}\varepsilon_jx_{j,m}\Big\|_{\infty}=\Big\|\displaystyle\sum_{j=0}^{n-1}x_{j,m}\Big\|_{\infty}& =1+\frac{\lambda_m}{\lambda_{m+n-1}}
+\frac{\lambda_{m+1}}{\lambda_{m+n-1}}+\ldots+\frac{\lambda_{m+n-2}}{\lambda_{m+n-1}}\\
& \rightarrow n \mbox{~as~} m\rightarrow\infty\\
&(\mbox{Since} \displaystyle\lim_{m\rightarrow\infty}\frac{\lambda_{m+i}}{\lambda_{m+n-1}}=1 \mbox{~for each~} i=0,1,2,\ldots,n-2).
\end{align*}
Again by definition, we have $J_n^s(\Lambda_\infty)=n$ and thus the theorem is proved.
\end{proof}
\begin{corollary}
$(i)$ Choose $\lambda_n=n+1$, then $J_n^s(ces_p)=n$ for $1<p\leq \infty$ (\cite{MAL1}).\\
$(ii)$ If $Q_n=\displaystyle\sum_{k=0}^{n}q_k=\lambda_n$, where $q=(q_k)=(\lambda_k-\lambda_{k-1})$ then $J_n^s(ces [p, q])=n$ for $1<p\leq \infty$.
\end{corollary}
\section{Geometric properties of $\Lambda_{\hat{p}}$}
Let $(X, \|.\|)$ be a Banach space and being a subspace of $l^0$. As usual, we denote $S(X)$ and $B(X)$ for the unit sphere and closed unit ball respectively. A point $x\in S(X)$ is said to be an extreme point of $B(X)$ if there does not exist two distinct points $y$, $z \in B(X)$ such that $2x=y+z$. The concept of extreme point plays an important role in the study of Krein-Milman theorem, Choquet integral representation theorem etc. \\
Foralewski \cite{FOR1} introduced the notion of \emph{coordinatewise Kadec-Klee property} of a Banach space and is denoted by \emph{($H_c$)}. $X$ is said to possess the property \emph{($H_c$)}, if $x\in X$ and every sequence $(x_l)\subset X$ such that
\begin{center}
 $\|x_l\|\rightarrow \|x\|$ and $x_{li}\rightarrow x_i$ as $l\rightarrow\infty$ for each $i$, then $\|x_l-x\|\rightarrow 0$.
\end{center}
If for every $\varepsilon>0$ there exists a $\delta>0$ such that
\begin{center}
$(x_l)\subset B(X)$, $sep(x_l)\geq\varepsilon$, $\|x_l\|\rightarrow \|x\|$ and $x_{li}\rightarrow x_i$ for each $i$ implies $\|x\|\leq1-\delta$,
\end{center}where $sep(x_l)=\inf\{\|x_l-x_m\|: l\neq m\}$, then we say $X$ has the \emph{coordinatewise uniformly Kadec-Klee property} and is denoted by $X\in (UKK_c)$ \cite{ZHA}. For any Banach space $X$, $(UKK_c)\Rightarrow (H_c)$.\\
$X$ is said to have the \emph{uniform Opial property} ($(UOP)$, in short) if for each $\varepsilon>0$ there exists $\mu>0$ such that
\begin{center}
$1+ \mu\leq \displaystyle\liminf_{l\rightarrow\infty}\|x_l+x\|$
\end{center}
for any weakly null sequence $(x_l)$ in $S(X)$ and $x\in X$ with $\|x\|\geq\varepsilon$ (see \cite{CUI5}, \cite{PRUS}).\\
A Banach space $X$ has the property $(\beta)$ if and only if, for every $\epsilon>0$, there exists $\delta>0$ such that, for each element $x\in B(X)$ and each sequence $(x_l)\in B(X)$ with $sep(x_l)\geq \epsilon$, there is an index $k$ such that
\begin{center}
$\Big\|\frac{x+x_k}{2}\Big\|\leq 1-\delta$ \cite{KUTZ}.
\end{center}
The \emph{weakly uniform normal structure} of a Banach space $X$ ($WUNS(X)$, in short) is determined by the \emph{weakly convergent sequence coefficient} of $X$ ($WCS(X)$, in short) (\cite{BENA}, \cite{BYNU}) is defined as
\begin{center}
$WCS(X)=\inf \Big\{\frac{\displaystyle\lim_k\sup_{n, m\geq k}\|x_n-x_m\|}{\inf\{\displaystyle\limsup_n\|x_n-y\|: y\in \mbox{Conv}(~x_n)\}}\Big\}$,
\end{center}
where infimum is taken over all weakly convergent sequence $(x_n)$ which is not norm convergent. If $WCS(X)>1$ then Banach space $X$ has $WUNS$. A Banach space $X$ has $WUNS$ if it possesses $(UOP)$ \cite{LIN}.\\
Let $(X_n, \|.\|_n)$ be Banach spaces for each $n\in \mathbb{N}_0$. Then the Nakano sequence spaces $l_{\hat{p}}(X_n)$ is defined as
\begin{center}
$l_{\hat{p}}(X_n)=\Big\{x=(x_n)_{n=0}^{\infty}: x_n\in X_n \mbox{~for each~} n\in \mathbb{N}_0 \mbox{~and~} \rho(rx)<\infty \mbox{~for some~}r>0\Big\},$
\end{center}
where convex modular $\rho$ is defined as $\rho(x)=\displaystyle\sum_{n=0}^{\infty}\|x_n\|_n^{p_n}$. It is easy to show that the sequence space $l_{\hat{p}}(X_n)$ is a Banach space equipped with the Luxemberg norm
\begin{center}
$\|x\|= \inf \Big\{ r>0 :  \rho(\frac{x}{r})\leq 1 \Big\}$.
\end{center}
Saejung \cite{SAEJ} proved that Ces\`{a}ro sequence spaces $ces_p$ for $1<p<\infty$ are isometrically embedded in the infinite $l_p$-sum $l_p(\mathbb{R}^n)$ of finite dimensional spaces $\mathbb{R}^n$. Here we present similar result for the sequence space $\Lambda_{\hat{p}}$.
\begin{lem}{\label{lemiso}}
The sequence space $\Lambda_{\hat{p}}$ is isometrically embedded in the Nakano sequence spaces $l_{\hat{p}}(\mathbb{R}^{n+1})$, where $\mathbb{R}^{n+1}$ is the ${(n+1)}$-dimensional Euclidean space equipped with the following norm:\\
$\|(\alpha_0, \alpha_1, \ldots, \alpha_{n})\|=\displaystyle\sum_{i=0}^{n}|\alpha_i|$ for $(\alpha_0, \alpha_1, \ldots, \alpha_{n})\in \mathbb{R}^{n+1}$.
\end{lem}
\begin{proof}
For all $x=(x_i)\in \Lambda_{\hat{p}}$, we define the following linear isometry $T: \Lambda_{\hat{p}}\rightarrow l_{\hat{p}}(\mathbb{R}^{n+1})$ by
\begin{center}
$T((x_i))=\Big(x_0, \Big(\frac{\lambda_0}{\lambda_1}x_0, \frac{(\lambda_1-\lambda_0)}{\lambda_1}x_1\Big), \ldots, \Big(\frac{\lambda_0}{\lambda_n}x_0, \frac{(\lambda_1-\lambda_0)}{\lambda_n}x_1, \ldots, \frac{(\lambda_n-\lambda_{n-1})}{\lambda_n}x_n\Big), \ldots\Big).$
\end{center}
Indeed
\begin{align*}
& \|T((x_i))\|_{l_{\hat{p}}(\mathbb{R}^{n+1})}\\
&=\|T(x_0, x_1, \ldots, x_i, \ldots)\|_{l_{\hat{p}}(\mathbb{R}^{n+1})}\\
&=\Big\|\Big(x_0, \Big(\frac{\lambda_0}{\lambda_1}x_0, \frac{(\lambda_1-\lambda_0)}{\lambda_1}x_1\Big), \ldots, \Big(\frac{\lambda_0}{\lambda_n}x_0, \frac{(\lambda_1-\lambda_0)}{\lambda_n}x_1, \ldots, \frac{(\lambda_n-\lambda_{n-1})}{\lambda_n}x_n\Big), \ldots\Big)\Big\|_{l_{\hat{p}}(\mathbb{R}^{n+1})}\\
&=\inf \Big\{ r>0 : \displaystyle\sum_{n=0}^{\infty}\Big(\frac{1}{r\lambda_n}\sum_{k=0}^{n}(\lambda_k-\lambda_{k-1})|x_k|\Big)^{p_n}\leq 1\Big\}\\
&=\inf \Big\{ r>
0 : \sigma\Big(\frac {x} {r}\Big) \leq 1 \Big\}\\
&=\|(x_i)\|_{\hat{p}}
\end{align*}Hence the lemma.
\end{proof}
Instead of studying geometric properties of $\Lambda_{\hat{p}}$ it is enough to study geometric properties of $l_{\hat{p}}(\mathbb{R}^{n+1})$ and if such geometric properties are inherited by subspaces then $\Lambda_{\hat{p}}$ will have the same property. Certain geometric structure of finite dimensional Banach spaces $l_p(X_n)$, $p>1$ is also investigated by Rolewicz \cite{ROLE}. \\
Saejung in his recent work (\cite{SAEJ}, Theorem 11 (2), p.535) established the following important result.
\begin{proposition}{\label{thmsaej}}
Suppose each $X_n$ is finite dimensional. Then the space $l_{\hat{p}}(X_n)$ has property $(\beta)$ and uniform Opial property if and only if $\displaystyle\limsup_{n\rightarrow \infty}p_n<\infty$.
\end{proposition}
Now we have the following new result.
\begin{theorem}
The sequence space $\Lambda_{\hat{p}}$ has property $(\beta)$ and uniform Opial property if and only if $\displaystyle\limsup_{n\rightarrow \infty}p_n<\infty$.
\end{theorem}
\begin{proof}
Since $\mathbb{R}^{n+1}$ is finite dimensional, Proposition \ref{thmsaej} implies that the space $l_{\hat{p}}(\mathbb{R}^{n+1})$ possesses the property $(\beta)$ and uniform Opial property if and only if $\displaystyle\limsup_{n\rightarrow \infty}p_n<\infty$. Since property $(\beta)$ and uniform Opial property are inherited by subspaces, so by Lemma \ref{lemiso} we obtain the result.
\end{proof}
\begin{corollary}
$(i)$ The space $\Lambda_{\hat{p}}$ has $WUNS$ if $\displaystyle\limsup_{n\rightarrow \infty}p_n<\infty$.\\
$(ii)$ If $\displaystyle\limsup_{n\rightarrow \infty}p_n<\infty$, and $\lambda_n=n+1$ then $ces_{\hat{p}}$ has property $(\beta)$ and uniform Opial property (\cite{SAEJ}, Theorem 11 (2)).
\end{corollary}
\begin{theorem}
The sequence space $\Lambda_{\hat{p}}$ possesses coordinate-wise uniform Kadec-Klee property.
\end{theorem}
\begin{proof}
Let $\varepsilon\in (0, 1)$ and take $\eta=(\frac{\varepsilon}{4})^{p^{*}}$, where $p^{*}=\displaystyle\sup_n p_n$. We choose $\delta\in (0, 1)$ such that $(1-\delta)^{p^{*}}>1-\eta$. Suppose $(x_l)\subset B(\Lambda_{\hat{p}})$, $sep(x_l)\geq\epsilon$, $\|x_l\|_{\hat{p}}\rightarrow \|x\|_{\hat{p}}$, $x_{il}\rightarrow x_i$ as $l\rightarrow\infty$ and for all $i\in \mathbb{N}_0$. We show that there exists a $\delta>0$ such that $\|x\|_{\hat{p}}\leq 1-\delta$. In contradict, we suppose that $\|x\|_{\hat{p}}> 1-\delta$. Then we can select a finite set $I=\{1, 2, \ldots, N-1\}$ on which $\|x|_I\|_{\hat{p}}> 1-\delta$. Since $x_{il}\rightarrow x_i$ for each $i\in \mathbb{N}_0$, therefore $x_l\rightarrow x$ uniformly on $I$. Consequently, since $\|x_l\|_{\hat{p}}\rightarrow \|x\|_{\hat{p}}$, there exists $l_N\in \mathbb{N}$ such that
\begin{center}
$\|x_l|_I\|_{\hat{p}}> 1-\delta$ and $\|(x_l-x_m)|_{I}\|_{\hat{p}}\leq \frac{\epsilon}{2}$ for all $l, m\geq l_N$.
\end{center}
From the first inequality, we simply get $\sigma(x_l|_I)\geq \|x_l|_I\|_{\hat{p}}^{p^{*}}> (1-\delta)^{p^{*}}> 1-\eta$ for $l\geq l_N$. Since $sep(x_l)\geq\epsilon$, i.e., $\|x_l-x_m\|_{\hat{p}}\geq\epsilon$, so the second inequality implies that $\|(x_l-x_m)|_{\mathbb{N}-I}\|_{\hat{p}}\geq \frac{\epsilon}{2}$ for $l, m\geq l_N, l\neq m$. Hence for $N\in \mathbb{N}$ there exists a $l_N$ such that $\|x_{l_N}|_{\mathbb{N}-I}\|_{\hat{p}}\geq \frac{\epsilon}{4}$. Without loss of generality, we may assume that $\|x_l|_{\mathbb{N}-I}\|_{\hat{p}}\geq \frac{\epsilon}{4}$ for all $l, N\in \mathbb{N}$. Therefore from the relation between norm and modular, we have $\sigma(x_l|_{\mathbb{N}-I})\geq \|x_l|_{\mathbb{N}-I}\|_{\hat{p}}^{p^{*}}\geq (\frac{\varepsilon}{4})^{p^{*}}=\eta$.\\
The convexity of the function $f(t)=|t|^{p_n}$ for each $n\in \mathbb{N}_0$, we have for any $\gamma\in [0,1]$ and $u\in \mathbb{R}$ that $f(\gamma u)=f(\gamma u+ (1-\gamma)0)\leq\gamma f(u)$. Therefore, if $0\leq u<v<\infty$, then $f(u)=f(\frac{u}{v}v)\leq \frac{u}{v}f(v)$, which means that $\frac{f(u)}{u}\leq \frac{f(v)}{v}$ for each $n\in \mathbb{N}_0$. Assuming now that $0\leq u,v<\infty$, $u+v>0$, we get
\begin{center}
$f(u+v)= u \frac{f(u+v)}{u+v} + v \frac{f(u+v)}{u+v}\geq u \frac{f(u)}{u}+ v \frac{f(v)}{v}=f(u)+ f(v)$
\end{center} for each $n\in \mathbb{N}_0$.
Since $x_l=x_l|_{I}+x_l|_{\mathbb{N}-I}$, applying the above fact we get $\sigma(x_l|_{I})+\sigma(x_l|_{\mathbb{N}-I})\leq\sigma(x_l)\leq 1$. It implies that $\sigma(x_l|_{\mathbb{N}-I})\leq 1-\sigma(x_l|_{I})< 1-(1-\eta)=\eta$, i.e., $\sigma(x_l|_{\mathbb{N}-I})<\eta$, which contradicts the inequality $\sigma_\Phi(x_l|_{\mathbb{N}-I})\geq \eta$ and this contradiction completes the proof.
\end{proof}

\begin{theorem}
Let $\displaystyle\limsup_{n\rightarrow \infty} p_n< \infty$. Then a point $x\in S(\Lambda_{\hat{p}})$ is an extreme point of $B(\Lambda_{\hat{p}})$ if and only if\\
$(i)$ $\sigma(x)=1$ and\\
$(ii)$ $Card(A_x)\leq 1$, \\
where $A_x=\{n\in Supp~ x: ~\Lambda x(n), ~\Lambda x(n+1), \ldots, ~\Lambda x(m-1)$ \mbox{~are different from zero and belongs}  \mbox{~to the interior of affine interval of~} $f(t)=|t|^{p_n}$ \mbox{for each} $n\in \mathbb{N}_0$, \mbox{~where~} $m$ \mbox{~is the smallest number~}\\ \mbox{~ such that~} $m>n$ \mbox{~and~} $m\in Supp~ x. \}$
\end{theorem}
\begin{proof}
We first prove the necessary part of the theorem. Let $x\in S(\Lambda_{\hat{p}})$ be an extreme point and condition $(i)$ is not true. Put $\varepsilon=1-\sigma(x)>0$ and consider the following two sequences:\\
$$y=(y_k)=(x_0, x_1, \ldots, x_{n_0}, 0, 0, \ldots)$$
$$z=(z_k)=(x_0, x_1, \ldots, x_{n_0}, 2x_{n_0+1}, 2x_{n_0+2}, \ldots).$$
Then it is clear that $2x=y+z$ and $y\neq z$. But
\begin{center}
$\sigma(y)=\displaystyle\sum_{n=0}^{n_0}\Big(\frac{1}{\lambda_n}\sum_{k=0}^{n}(\lambda_k-\lambda_{k-1})|x_k|\Big)^{p_n}\leq \sigma(x)=1-\varepsilon<1,$ \mbox{~and~}
\end{center}
\begin{align}{\label{eqnsigmaz}}
\sigma(z)&\leq\displaystyle\sum_{n=0}^{n_0}\Big(\frac{1}{\lambda_n}\sum_{k=0}^{n}(\lambda_k-\lambda_{k-1})|x_k|\Big)^{p_n}+ \displaystyle\sum_{n=n_0+1}^{\infty}2^{p_n}\Big(\frac{1}{\lambda_n}\sum_{k=0}^{n}(\lambda_k-\lambda_{k-1})|x_k|\Big)^{p_n}
\end{align}
Since $\displaystyle\limsup_{n\rightarrow \infty} p_n< \infty$ and $x\in \Lambda_{\hat{p}}$, so $\exists$ a natural number $n_0$ and a constant $M>0$ such that for all $n>n_0$, we have $2^{p_n}\leq 2^M$ and for every $\varepsilon>0$,  $\displaystyle\sum_{n=n_0+1}^{\infty}\Big(\frac{1}{\lambda_n}\sum_{k=0}^{n}(\lambda_k-\lambda_{k-1})|x_k|\Big)^{p_n}<\frac{\varepsilon}{2^M}$.
Hence from Eqn. $(\ref{eqnsigmaz})$, we obtain $\sigma(z)<\sigma(x)+\varepsilon=1$. The relation between norm and modular implies that $\|y\|_{\hat{p}}\leq 1$ and $\|z\|_{\hat{p}}\leq 1 $, which contradicts to the assumption that $x$ is an extreme point. Therefore $\sigma(x)=1$, i.e., condition $(i)$ is proved.\\
To prove condition $(ii)$, if possible, we assume that $A_x$ has at least two elements and show that $x$ is not an extreme point. Let $n_1$, $n_2\in A_x$ and $n_1\neq n_2$. We denote by $m_1$ (respectively $m_2$) the smallest number in $Supp~x$ such that $m_1>n_1$ ( respectively $m_2>n_2$) and assume that $m_1\leq n_2$. Therefore by assumption, we have $\Lambda x(n_1), ~\Lambda x(n_1+1), \ldots, ~\Lambda x(m_1-1)$ and $\Lambda x(n_2), ~\Lambda x(n_2+1), \ldots, ~\Lambda x(m_2-1)$ are different from zero and belongs to interior of the affine intervals of $f(t)=|t|^{p_n}$, $n\in \mathbb{N}_0$. Recall that $\Lambda x(n)=\frac{1}{\lambda_n}\displaystyle\sum_{k=0}^{n}(\lambda_k-\lambda_{k-1})|x_k|$. The affine function defined on these intervals by the formula $f(t)=a_jt+b_j$, where $j=n_1, n_1+1, \ldots, m_1-1, n_2, \ldots, m_2-1$. Denote
\begin{align*}
\varepsilon_1&= \frac{a_{n_1}}{\lambda_{n_1}}+ \frac{a_{n_1+1}}{\lambda_{n_1+1}}+\ldots+\frac{a_{m_1-1}}{\lambda_{m_1-1}} & \mbox{~and~}
\varepsilon_2&= \frac{a_{n_2}}{\lambda_{n_2}}+ \frac{a_{n_2+1}}{\lambda_{n_2+1}}+\ldots+\frac{a_{m_2-1}}{\lambda_{m_2-1}}.
\end{align*}
Choose $\delta_1\neq 0$ and $\delta_2\neq 0$ such that
$\varepsilon_1\delta_1=\varepsilon_2\delta_2$ and
\begin{center}
$\Lambda x(i)\pm \frac{\delta_1}{i}$, $i=n_1, n_1+1, \ldots, m_1-1$ and $\Lambda x(j)\pm \frac{\delta_2}{j}$, $j=n_2, \ldots, m_2-1$
\end{center} are all belongs to affine intervals of $f$.\\
Consider two sequences $y=(y_n)$ and $z=(z_n)$ defined as follows when $m_1<n_2$:
\begin{align*}
&y=(x_0, \ldots, x_{n_1-1}, x_{n_1}+\frac{sgn~x_{n_1}}{\lambda_{n_1}-\lambda_{n_1-1}}\delta_{1}, x_{n_1+1}, \ldots, x_{m_1-1}, x_{m_1}-\frac{sgn~x_{m_1}}{\lambda_{m_1}-\lambda_{m_1-1}}\delta_{1}, x_{m_1+1}, \\
&~~~\ldots, ~x_{n_2-1}, x_{n_2}- \frac{sgn~x_{n_2}}{\lambda_{n_2}-\lambda_{n_2-1}}\delta_{2}, x_{n_2+1}, \ldots, x_{m_2-1}, x_{m_2}+\frac{sgn~x_{m_2}}{\lambda_{m_2}-\lambda_{m_2-1}}\delta_{2}, x_{m_2+1}, \ldots)\\
&z=(x_0, \ldots, x_{n_1-1}, x_{n_1}-\frac{sgn~x_{n_1}}{\lambda_{n_1}-\lambda_{n_1-1}}\delta_{1}, x_{n_1+1}, \ldots, x_{m_1-1}, x_{m_1}+\frac{sgn~x_{m_1}}{\lambda_{m_1}-\lambda_{m_1-1}}\delta_{1}, x_{m_1+1},\\
&~~~\ldots, ~x_{n_2-1}, x_{n_2}+\frac{sgn~x_{n_2}}{\lambda_{n_2}-\lambda_{n_2-1}}\delta_{2}, x_{n_2+1}, \ldots, x_{m_2-1}, x_{m_2}-\frac{sgn~x_{m_2}}{\lambda_{m_2}-\lambda_{m_2-1}}\delta_{2}, x_{m_2+1}, \ldots),
\end{align*}where $sgn$ denotes the signum function.
The case when $m_1=n_2$, we define $y$ and $z$ similarly, where
\begin{align*}
y_{m_1}&= x_{m_1}-\frac{sgn~x_{m_1}}{\lambda_{m_1}-\lambda_{m_1-1}}(\delta_1+\delta_2)
& \mbox{~and~~}
z_{m_1}&= x_{m_1}+\frac{sgn~x_{m_1}}{\lambda_{m_1}-\lambda_{m_1-1}}(\delta_1+\delta_2).
\end{align*}
For the above two sequences $y$ and $z$, we have $y+z=2x$ and \\
\begin{align*}
\displaystyle\sum_{n=n_1}^{m_1-1}(\Lambda y(n))^{p_n}&= \displaystyle\sum_{n=n_1}^{m_1-1}(\Lambda x(n))^{p_n}+\varepsilon_1\delta_1,~~~\displaystyle\sum_{n=n_2}^{m_2-1}(\Lambda y(n))^{p_n}= \displaystyle\sum_{n=n_2}^{m_2-1}(\Lambda x(n))^{p_n}-\varepsilon_2\delta_2
\end{align*}
\begin{align*}
\displaystyle\sum_{n=n_1}^{m_1-1}(\Lambda z(n))^{p_n}&= \displaystyle\sum_{n=n_1}^{m_1-1}(\Lambda x(n))^{p_n}-\varepsilon_1\delta_1,~~~\displaystyle\sum_{n=n_2}^{m_2-1}(\Lambda z(n))^{p_n}= \displaystyle\sum_{n=n_2}^{m_2-1}(\Lambda x(n))^{p_n}+\varepsilon_2\delta_2.
\end{align*}
Hence $\sigma(y)=\sigma(z)=\sigma(x)=1$ and relation between norm and modular implies that $\|x\|=1$, $\|y\|=1$ and $\|z\|=1$. Therefore $x$ is not an extreme point and our assumption that $A_x$ has atleast two elements is wrong. So $Card(A_x)\leq 1$, i.e, condition $(ii)$ is proved.\\
Now we proof sufficiency of the theorem, that is if $x\in S(\Lambda_{\hat{p}})$ then $x$ is an extreme point of $B(\Lambda_{\hat{p}})$. Let the assumptions $(i)-(ii)$ holds. We assume that there exists $y$, $z$ $\in S(\Lambda_{\hat{p}})$ such that $x=\frac{y+z}{2}$ but $y\neq z$. Let $n_1\in \mathbb{N}$ be the smallest number such that $y_{n_1}\neq z_{n_1}$. Then we show that $n_1\in Supp x$. Suppose in contradict, we assume that $n_1\notin Supp x$, that is $x_{n_1}=0$. Then we immediately have $|y_{n_1}|=|z_{n_1}|>0$. Let $l$ be the smallest number in $Supp x$ such that $n_1<l$. With out loss of generality, we assume that $|x_{l}|\leq |y_{l}|$ and we have $0=|x_{n}|\leq |y_{n}|$ for $n={n_1+1}, \ldots, l-1$. Now
\begin{align*}
&\Lambda{y(l)}\\
&= \frac{1}{\lambda_l}(\lambda_0|y_{0}|+\ldots+(\lambda_{n_1}-\lambda_{n_1-1})|y_{n_1}|+
(\lambda_{n_1+1}-\lambda_{n_1})|y_{n_1+1}|+\ldots+(\lambda_{l}-\lambda_{l-1})|y_{l}|)\\
&= \frac{1}{\lambda_l}(\lambda_0|x_{0}|+\ldots+(\lambda_{n_1}-\lambda_{n_1-1})|y_{n_1}|+
(\lambda_{n_1+1}-\lambda_{n_1})|y_{n_1+1}|+\ldots+(\lambda_{l}-\lambda_{l-1})|y_{l}|)\\
& \geq \frac{1}{\lambda_l}(\lambda_0|x_{0}|+\ldots+(\lambda_{n_1}-\lambda_{n_1-1})|y_{n_1}|+
(\lambda_{n_1+1}-\lambda_{n_1})|x_{n_1+1}|+\ldots+(\lambda_{l}-\lambda_{l-1})|x_{l}|)\\
&> \frac{1}{\lambda_l}(\lambda_0|x_{0}|+\ldots+(\lambda_{n_1}-\lambda_{n_1-1})|x_{n_1}|+
(\lambda_{n_1+1}-\lambda_{n_1})|x_{n_1+1}|+\ldots+(\lambda_{l}-\lambda_{l-1})|x_{l}|)\\
&=\Lambda{x(l)},
\end{align*}
which implies that $\sigma(y)>\sigma(x)=1$, a contradiction with the assumption that $y\in S(\Lambda_{\hat{p}})$. Therefore $n_1\in Supp x$.
Since
\begin{align*}
1&=\sigma(x)=\sigma\Big(\frac{y+z}{2}\Big)=\displaystyle\sum_{n=0}^{\infty}\Big(\frac{1}{\lambda_n}\sum_{k=0}^{n}(\lambda_k-\lambda_{k-1})\Big|\frac{y_k+z_k}{2}\Big|
\Big)^{p_n}\\
& \leq \displaystyle\sum_{n=0}^{\infty}\Big(\frac{1}{2\lambda_n}\sum_{k=0}^{n}(\lambda_k-\lambda_{k-1})|y_k|+
\frac{1}{2\lambda_n}\sum_{k=0}^{n}(\lambda_k-\lambda_{k-1})|z_k|\Big)^{p_n}\\
& \leq \frac{1}{2}\displaystyle\sum_{n=0}^{\infty}\Big(\frac{1}{\lambda_n}\sum_{k=0}^{n}(\lambda_k-\lambda_{k-1})|y_k|\Big)^{p_n}+ \frac{1}{2}\displaystyle\sum_{n=0}^{\infty}\Big(\frac{1}{\lambda_n}\sum_{k=0}^{n}(\lambda_k-\lambda_{k-1})|z_k|\Big)^{p_n}\\
& = \frac{1}{2}\big(\sigma(y)+\sigma(z)\big)=1.
\end{align*}
Therefore for each $n\in \mathbb{N}_0$, we have
\begin{align}{\label{eqnequal}}
&\displaystyle\sum_{n=0}^{\infty}\Big(\frac{1}{\lambda_n}\sum_{k=0}^{n}(\lambda_k-\lambda_{k-1})\Big|\frac{y_k+z_k}{2}\Big|
\Big)^{p_n} \nonumber
\end{align}
\begin{align}
&=\frac{1}{2}\displaystyle\sum_{n=0}^{\infty}\Big(\frac{1}{\lambda_n}\sum_{k=0}^{n}(\lambda_k-\lambda_{k-1})|y_k|\Big)^{p_n}+ \frac{1}{2}\displaystyle\sum_{n=0}^{\infty}\Big(\frac{1}{\lambda_n}\sum_{k=0}^{n}(\lambda_k-\lambda_{k-1})|z_k|\Big)^{p_n}
\end{align}
With out loss of generality, we assume that $0\leq |y_{n_1}|<|x_{n_1}|<|z_{n_1}|$. Using the strict convexity of the function $f(t)=|t|^{p_n}$, $n\in \mathbb{N}_0$, we have
\begin{align}{\label{eqninequal}}
(\Lambda y(n_1))^{p_{n_1}}<(\Lambda z(n_1))^{p_{n_1}}.
\end{align}
Let $m_1\in \mathbb{N}_0$ be the smallest number such that $m_1>n_1$ and $m_1\in \mbox{Supp}~ x$. If $x_n=0$ for every $n>n_1$ then by Eqn (\ref{eqninequal}), we arrived at a contradiction $\sigma(y)<\sigma(z)$. For $n=n_1, n_1+1, \ldots, m_1-1$, we have $\Lambda y(n)<\Lambda z(n)$. If there exists $n\in\{n_1, n_1+1, \ldots, m_1-1\}$ such that
\begin{center}
$(\Lambda x(n))^{p_n}< \frac{1}{2}\Big((\Lambda y(n))^{p_n}+(\Lambda z(n))^{p_n}\Big)$
\end{center}
then we get a contradiction with Eqn (\ref{eqnequal}). Therefore for every $n\in\{n_1, n_1+1, \ldots, m_1-1\}$ we assume that
\begin{center}
$(\Lambda x(n))^{p_n}= \frac{1}{2}\Big((\Lambda y(n))^{p_n}+(\Lambda z(n))^{p_n}\Big).$
\end{center}
Then by assumption, we have $n_1\in A_x$. Since $\sigma(y)=1$, $\sigma(z)=1$, so by Eqn (\ref{eqninequal}), there exists $n\geq m_1$ such that $\Lambda y(n)> \Lambda z(n)$. Let $n_2$ be the such a smallest number by which $\Lambda z(n_2)< \Lambda y(n_2)$. Then $|z_{n_2}|<|y_{n_2}|$ and so $n_2\in \mbox{Supp}~x$. If $x_n=0$ for every $n>m_2$ then we get a contradiction with $\sigma(z)<\sigma(y)$.\\
Let $m_2$ be the smallest number in $\mbox{Supp}~x$ such that $m_2>n_2$. By assumption $n_2\notin A_x$, so there exists $n\in  \{n_2, n_2+1, \ldots, m_2-1\}$ such that
\begin{center}
$(\Lambda x(n))^{p_n}< \frac{1}{2}\Big((\Lambda y(n))^{p_n}+(\Lambda z(n))^{p_n}\Big),$
\end{center}
which contradicts the Eqn (\ref{eqnequal}). Thus we must have $x=y=z$, i.e., $x$ is an extreme point.
\end{proof}

\section{Von Neumann-Jordan constant of $\Lambda_2^{(2)}$}
For two dimensional sequence space $\Lambda_p^{(2)}$, norm $\|(u, v)\|_p$ is given by
$$\|(u, v)\|_p=\Big(|u|^p+\Big(\frac{\lambda_0|u|+(\lambda_1-\lambda_0)|v|}{\lambda_1}\Big)^p\Big)^{\frac{1}{p}}.$$
The \emph{von Neumann-Jordan constant} $C_{NJ}(X)$ of a Banach space $X$ was introduced by Clarkson \cite{CLARK} and is defined as follows:
\begin{center}
$C_{NJ}(X)=\sup\Big\{\frac{\|x+ y\|^2+ \|x-y\|^2}{2(\|x\|^2+ \|y\|^2)}: x, y\in X \mbox{~and~}\|x\|+ \|y\|\neq 0\Big\}.$
\end{center}
Now we begin with the following result.
\begin{theorem}
The von Neumann-Jordan constant $C_{NJ}(\Lambda_2^{(2)})=1+\frac{\lambda_0}{\sqrt{\lambda_0^2+\lambda_1^2}}$.
\end{theorem}
\begin{proof}
Choose $(a, b)$, $(c, d)$ $\in \Lambda_2^{(2)}$. Then
\begin{align*}
\|(a, b)\pm (c, d)\|^2 &= \|(a\pm c, b\pm d)\|^2\\
&=\Big(|a\pm c|^2+\Big(\frac{\lambda_0|a\pm c|+(\lambda_1-\lambda_0)|b\pm d|}{\lambda_1}\Big)^2\Big)\\
&=\Big(1+\frac{\lambda_0^2}{\lambda_1^2}\Big)|a\pm c|^2+\frac{2\lambda_0(\lambda_1-\lambda_0)}{\lambda_1^2}|a\pm c||b\pm d|+\frac{(\lambda_1-\lambda_0)^2}{\lambda_1^2}|b\pm d|^2.
\end{align*}
Now
\begin{align*}
\frac{2\lambda_0(\lambda_1-\lambda_0)}{\lambda_1^2}|a\pm c||b\pm d| &=\frac{\lambda_0}{\sqrt{\lambda_0^2+\lambda_1^2}}\Big\{2\frac{\sqrt{\lambda_0^2+\lambda_1^2}}{\lambda_1}|a\pm c|\frac{(\lambda_1-\lambda_0)}{\lambda_1}|b\pm d|\Big\}\\
&\leq \frac{\lambda_0}{\sqrt{\lambda_0^2+\lambda_1^2}}\Big\{\frac{\lambda_0^2+\lambda_1^2}{\lambda_1^2}|a\pm c|^2+\frac{(\lambda_1-\lambda_0)^2}{\lambda_1^2}|b\pm d|^2\Big\}.
\end{align*}
Therefore
\begin{align*}
&\|(a, b)+ (c, d)\|^2+ \|(a, b)-(c, d)\|^2\\
&\leq \Big(1+\frac{\lambda_0}{\sqrt{\lambda_0^2+\lambda_1^2}}\Big)\Big\{\frac{\lambda_0^2+\lambda_1^2}{\lambda_1^2}(|a+c|^2+|a-c|^2)+
\frac{(\lambda_1-\lambda_0)^2}{\lambda_1^2}(|b+d|^2+|b-d|^2)\Big\}\\
&=2\Big(1+\frac{\lambda_0}{\sqrt{\lambda_0^2+\lambda_1^2}}\Big)\Big\{\frac{\lambda_0^2+\lambda_1^2}{\lambda_1^2}(|a|^2+|c|^2)+
\frac{(\lambda_1-\lambda_0)^2}{\lambda_1^2}(|b|^2+|d|^2)\Big\}\\
&\leq 2\Big(1+\frac{\lambda_0}{\sqrt{\lambda_0^2+\lambda_1^2}}\Big)\Big\{\frac{\lambda_0^2+\lambda_1^2}{\lambda_1^2}|a|^2+ \frac{2\lambda_0(\lambda_1-\lambda_0)}{\lambda_1^2}|a||b|+\frac{(\lambda_1-\lambda_0)^2}{\lambda_1^2}|b|^2\\
& ~~~~+ \frac{\lambda_0^2+\lambda_1^2}{\lambda_1^2}|c|^2+ \frac{2\lambda_0(\lambda_1-\lambda_0)}{\lambda_1^2}|c||d|+\frac{(\lambda_1-\lambda_0)^2}{\lambda_1^2}|d|^2\Big\}\\
&=\Big(1+\frac{\lambda_0}{\sqrt{\lambda_0^2+\lambda_1^2}}\Big)\big\{2(\|(a, b)\|^2+ \|(c, d)\|^2)\big\}.
\end{align*}
Hence $C_{NJ}(\Lambda_2^{(2)})\leq 1+\frac{\lambda_0}{\sqrt{\lambda_0^2+\lambda_1^2}}$.\\
Consider $(a, b)=(\lambda_1-\lambda_0, 0)$ and $(c, d)=(0, \sqrt{\lambda_0^2+\lambda_1^2})$. Then from the definition, we have
$$C_{NJ}(\Lambda_2^{(2)})\geq \frac{\|(a, b)+ (c, d)\|^2+ \|(a, b)-(c, d)\|^2}{2(\|(a, b)\|^2+ \|(c, d)\|^2)}=1+\frac{\lambda_0}{\sqrt{\lambda_0^2+\lambda_1^2}}.$$
Combining last two inequalities, we obtain $C_{NJ}(\Lambda_2^{(2)})=1+\frac{\lambda_0}{\sqrt{\lambda_0^2+\lambda_1^2}}$.
\end{proof}
\begin{corollary}
The James constant $J(\Lambda_2^{(2)})=\sqrt{2+\frac{2\lambda_0}{\sqrt{\lambda_0^2+\lambda_1^2}}}$.
\end{corollary}
\begin{proof}
It is well-known that $\frac{1}{2}\big(J(\Lambda_2^{(2)})\big)^2\leq C_{NJ}(\Lambda_2^{(2)})=1+\frac{\lambda_0}{\sqrt{\lambda_0^2+\lambda_1^2}}$. Therefore $J(\Lambda_2^{(2)})\leq\sqrt{2+\frac{2\lambda_0}{\sqrt{\lambda_0^2+\lambda_1^2}}}$. Equality occurs when $x=\Big(\frac{\lambda_1}{\sqrt{\lambda_0^2+\lambda_1^2}}, 0\Big)$ and $y=\Big(0, \frac{\lambda_1}{\lambda_1-\lambda_0}\Big)$.
\end{proof}
\begin{corollary}
If $\lambda_0=1$ and $\lambda_1=2$ then $C_{NJ}(ces_2^{(2)})=1+\frac{1}{\sqrt{5}}$ and $J(ces_2^{(2)})=\sqrt{2+\frac{2}{\sqrt{5}}}$ (see \cite{MAL1}, \cite{SAEJ}).
\end{corollary}
\begin{corollary}
If $\lambda_0=q_0$ and $\lambda_1=q_0+q_1$ then $C_{NJ}(ces[2, q]^{(2)})=1+\frac{q_0}{\sqrt{2q_0^2+2q_0q_1+q_1^2}}$ and $J(ces[2, q]^{(2)})=\sqrt{2+\frac{2q_0}{\sqrt{2q_0^2+2q_0q_1+q_1^2}}}$.
\end{corollary}
The following theorem is a counter part of a theorem presented in (\cite{SAEJ}, Theorem 15, p.536) for the sequence space $\Lambda_p^{(2)}$. We repeat here for the sake of completeness.
\begin{theorem}
The von Neumann-Jordan constant $C_{NJ}(\Lambda_p^{(2)})=\Big(\displaystyle\sup_{0\leq t\leq 1}\frac{\psi(t)}{\psi_2(t)}\Big)^2$, $1<p\leq 2$ where\\
$\psi(t)=\Big(\frac{\lambda_1^p(1-t)^p}{\lambda_0^p+\lambda_1^p}+\Big(\frac{\lambda_0(1-t)}{(\lambda_0^p+\lambda_1^p)^{1/p}}+t\Big)^p\Big)^{\frac{1}{p}}$ and
$\psi_2(t)=\sqrt{\big\{(1-t)^2+t^2\big\}}$.
\end{theorem}
\begin{proof}
For $(x, y)\in \mathbb{R}^2$, we define a norm on $\mathbb{R}^2$ as
\begin{center}
$|(x, y)|=\Big\|\Big(\frac{\lambda_1x}{(\lambda_0^p+\lambda_1^p)^{1/p}}, \frac{\lambda_1y}{(\lambda_1-\lambda_0)}\Big)\Big\|_{\Lambda_p^{(2)}}$.
\end{center}
It is easy to verify that $|(x, y)|=|(|x|, |y|)|$ and $|(1, 0)|=1=|(0, 1)|$, that is $|(x, y)|$ defines an absolute and normalized norm.
Additionally, choose a map $T: (\mathbb{R}^2, |(,)|)\rightarrow \Lambda_p^{(2)}$ defined as
\begin{center}
$T\big((x, y)\big)=\Big(\frac{\lambda_1x}{(\lambda_0^p+\lambda_1^p)^{1/p}}, \frac{\lambda_1y}{(\lambda_1-\lambda_0)}\Big)$.
\end{center}
It is easy to show that $T$ is an isometric isomorphism, i.e, sequence spaces $\Lambda_p^{(2)}$ are isometrically isomorphic to $(\mathbb{R}^2, |(,)|)$. Using Saito et al. result (\cite{SAITO}, Theorem 1, p. 521) it is enough to prove that $\psi\geq \psi_2$ and observe that
\begin{center}
$\Big(\frac{\lambda_0(1-t)}{(\lambda_0^p+\lambda_1^p)^{1/p}}+t\Big)^p \geq \frac{\lambda_0^p(1-t)^p}{\lambda_0^p+\lambda_1^p}+t^p$,
\end{center}
which implies $\psi(t)\geq \big\{(1-t)^p+t^p\big\}^{\frac{1}{p}}\geq \big\{(1-t)^2+t^2\big\}^{\frac{1}{2}}=\psi_2(t)$. Hence the result.
\end{proof}

\end{document}